\documentclass[10pt,reqno]{amsart}
\usepackage{graphicx}
\usepackage{indentfirst,csquotes,mathdots}
\usepackage{amssymb,amsthm,amsmath}
\usepackage{xcolor,paralist,titlesec,fancyhdr,etoolbox} 

\usepackage{hyperref} 
\hypersetup{ colorlinks=true, linkcolor=black, filecolor=black, urlcolor=black }

\usepackage[capitalize]{cleveref}

\usepackage{aliascnt}

\newaliascnt{definition}{theorem}

\aliascntresetthe{definition}
\crefname{definition}{Definition}{Definitions}

\newaliascnt{example}{theorem}

\aliascntresetthe{example}
\crefname{example}{Example}{Examples}

\newaliascnt{lemma}{theorem}
\newtheorem{lemma}[lemma]{Lemma}
\aliascntresetthe{lemma}
\crefname{lemma}{Lemma}{Lemmas}

\newaliascnt{proposition}{theorem}
\newtheorem{proposition}[proposition]{Proposition}
\aliascntresetthe{proposition}
\crefname{proposition}{Proposition}{Propositions}

\newaliascnt{corollary}{theorem}
\newtheorem{corollary}[corollary]{Corollary}
\aliascntresetthe{corollary}
\crefname{corollary}{Corollary}{Corollaries}

\newaliascnt{conjecture}{theorem}

\aliascntresetthe{conjecture}
\crefname{conjecture}{Conjecture}{Conjectures}
\titleformat{\section}[hang]{\normalfont\Large\bfseries}{\thesection}{1em}{}
\titleformat{\subsection}[hang]{\normalfont\large\bfseries}{\thesubsection}{1em}{}

\usepackage{lipsum}

\begin{document}
\title{Properties of Kloosterman Sums for Long Weyl Element on GL\((n)\) } 
\author[]{Srijeet Bhattacharjee}
\date{}
\address{Indian Statistical Institute, 203 Barrackpore,
Trunk Road, Kolkata 700108, West Bengal, India \vspace{0.5cm}}
\email{srijeet.b2005@mail.com}
\maketitle

\let\thefootnote\relax
\footnotetext{} 

\begin{abstract}
We give an alternate expression for the Kloosterman sums associated with the long Weyl element of \(GL(n)\) that avoides the use of Pl\"ucker coordinates. The expression is instead based on the use of the well-known Cauchy-Binet formula from linear algebra. From this expression, we prove some important properties of these Kloosterman sums.
\end{abstract} 
\noindent
\section{Introduction}
The classical Kloosterman sums were introduced by Kloosterman in the study of representation of integers by positive quadratic forms, through Kloosterman's refinement of the Hardy-Littlewood circle method. Given integers \(m,n\) and a modulus \(c>0\), this is given by an algebraic exponential sum

\begin{equation}
S(m,n;c)=\sum_{\substack{a,d \pmod{c}\\ad\equiv1 \pmod{c}}}e\left(\frac{ma+nd}{c}\right)\label{explicit2}
\end{equation}

Subsequently, these sums were realized as Fourier coefficients of Poincar\'e series, whence they appear in the geometric side of the Petersson/Bruggeman-Kuznetsov trace formula. Bounds on Kloosterman sums are important in analytic number theory because they give information about the analytic properties of modular or Maa\ss\ forms.

Using factorisation properties, Kloosterman sums can be broken down into prime power modulus. Andr\'e Weil used techniques from algebraic geometry to prove square root cancelation of these sums and combining the prime power factors one can show that,
\[
S(m,n;c)\le \tau(c)c^{1/2}(m,n,c)^{1/2}
\]

where \(\tau\) is the divisor function. In the higher rank setup of automorphic forms, expressions of Poincar\'e series and Kloosterman series were first carried out for \(GL(3)\) in \cite{BFG88}. Corresponding to each element of the Weyl group of \(GL(3)\), there is an associated Kloosterman sum. The most important sum (due to it's contribution in the higher rank Kuznetsov trace formula) is associated to the long Weyl element, \(w_0=\begin{pmatrix}
    &&1\\&-1\\1
\end{pmatrix}\), and the corresponding Kloosterman sum has the following expression for \(S(M_1,M_2,N_1,N_2,D_1,D_2)\),
\begin{equation}
\sum_{\substack{B_1,C_1 \pmod{D_1}\\B_2,C_2 \pmod{D_2}\\(B_i,C_i,D_i)=1\\D_1D_2|D_1C_2+B_1B_2+C_1D_2}}e\left(\frac{M_1B_1+N_1(Y_1D_2-Z_1B_2)}{D_1}+\frac{M_2B_2+N_2(Y_2D_1-Z_2B_1)}{D_2}\right)
\label{explicit3}
\end{equation}

where \(Y_i,Z_i\) are chosen such that \(Y_iB_i+Z_iC_i\equiv1\pmod{D_i}\) for \(i=1,2\). 

The explicit form of the Kloosterman sum was obtained by using Pl\"ucker coordinates to understand the Bruhat decomposition of \(GL(3)\). However, as pointed out in \cite{BM24}, this approach does not extend to \(n\ge 3\) because the number of Pl\"ucker relations grows rapidly with \(n\). [\cite{G06}, Ch-11] also gives an account of the general theory of Poincar\'e series, Kloosterman sums, Pl\"ucker coordinates and Kuznetsov trace formula for \(GL(n)\).

We give an alternate expression for the Kloosterman sum associated to the long Weyl element for \(GL(n)\), using the well-known Cauchy-Binet formula for the determinant of a sub-matrix of the product of two matrices \(C=AB\) involving the determinants of the sub-matrices of the individual matrices \(A\) and \(B\). The expression is not as explicit as \eqref{explicit3} but from this expression we can derive almost all the properties of Kloosterman sums mentioned in \cite{BFG88} without invoking Pl\"ucker coordinates. 

\subsection{Notations}
    For tuple \(T\) of length \(n\), \(T_i\) will denote the \(i\)-th coordinate of \(T\), \(\tilde{T}=(T_{n},T_{n-1},...,T_1)\) will denote the components of \(T\) reversed, and for any \(1\le i\le n\), \(T^{(i)}=(T_1,\cdots T_{i-1},T_{i+1},\cdots,T_n)\). For \(T,\ P\in \mathbb{Z}^{n},\ TP\) denotes usual pointwise product of \(T\) and \(P\).
    
    For \(A\in M(n,\mathbb{R}),\ A_{i,j}\) will denote \((i,j)\)-th coordinate of \(A\), \(\Lambda_{i_1,\cdots,i_m}^{j_1,\cdots,j_m}(A)\) will denote the determinant of the \(m\times m\) matrix formed by rows indexed by \(i_1,\cdots,i_m\) and columns indexed by \(j_1,\cdots,j_n\). \(A^t\) will denote the transpose of \(A\) and \(A^{at}\) will denote the ``anti-transpose'' of \(A\), obtained by reflecting the elements of \(A\) along the anti-diagonal of \(A\).
    
\section{Preliminaries}
We recall the definition of Kloosterman sums. Let \(G\) be the reductive group, GL\((n)\) of invertible \(n\times n\) matrices. \(T\) denotes the maximal torus of diagonal matrices with non-zero entries in \(G\), \(U\) denotes the standard maximal unipotent subgroup of upper triangular matrices with all diagonal entries being \(1\). Let \(N=N_G(T)\) be the normalizer of \(T\) in \(G\), which consists of matrices that have exactly one non-zero entry in each row and column. 

\(W=N/T\) is called the Weyl Group. Each equivalence class has a unique permutation matrix as a representative, and hence, for \(G\), \(W\cong S_n\), the Symmetric group on \(n\) elements. However, it is a standard practice to change the sign of some of the entries to force the representative to be in SL\((n)\).  For such a representative \(w\), let \(U_w= U \cap w^{-1}U^tw\). Consider a congruence subgroup \(\Gamma\) of \(G(\mathbb{Z})\), and an element \(c\in N(\mathbb{R})\). Let \(C(c)=U(\mathbb{R})cwU(\mathbb{R})\cap \Gamma,\ X(c)=U(\mathbb{Z})\backslash C(n)/U_w(\mathbb{Z})\), and define the canonical projections \(\mu: X(c)\to U(\mathbb{Z})\backslash U(\mathbb{R}),\ \mu':X(c)\to U(\mathbb{R})/U(\mathbb{Z})\).

Consider any two characters \(\psi,\psi'\) of \(U(\mathbb{R})\) that are trivial on \(U(\mathbb{Z})\). It is well known that any such character is of the form \(\psi_N\), where \(N=(N_1,...,N_{n-1})\in \mathbb{Z}^{n-1}\) and
\[
\psi\left(\begin{pmatrix}
1 & x_{1,2} & x_{1,3}& \cdots & x_{1,n} \\
   & 1 & x_{2,3}&\cdots & x_{2,n} \\
   && 1 & \cdots & x_{3,n} \\
   & & & \ddots & \vdots \\
  & &  &  & 1
\end{pmatrix}\right)=e(N_1x_{1,2}+N_2x_{2,3}+...+N_{n-1}{x_{(n-1),n}})
\]
The Kloosterman sum is defined as
\[
Kl(\psi,\psi',c)=\sum_{\substack{x\in X(c)}}\psi(\mu(x))\psi'(\mu'(x))
\]
In particular, given \(M,\ N,\ D\in \mathbb{Z}^{n-1}\)  if \(c=\)Diag\((1/D_{n-1},D_{n-2}/D_{n-1},...,D_2/D_1,D_1)\), \(\psi=\psi_{\tilde{M}},\ \psi'=\psi_{N}\), the Kloosterman sum obtained appears as the Fourier coefficient \(S_n(M,N;D,w)\) of the GL\((n)\) Poincar\'e Series. One can check [See \cite{G06}] that for \(n=2\), and \(w=\begin{pmatrix}
    &-1\\1
\end{pmatrix}\), the only non-trivial Weyl element for \(GL(2)\), we get the classical Kloosterman sum given in \cref{explicit2}.

It is known that the Kloosterman is zero unless \(w\) is a block anti-diagonal matrix of the form \(w_0=\begin{pmatrix}
    &&&I_{i_1}\\&&\iddots\\&I_{i_{k-1}}\\I_{i_k}
\end{pmatrix}\) with \(I_{i_1}+\cdots+I_{i_{k-1}}+I_{i_k}=n\).
One extreme case for the Weyl element is given by \(w=\begin{pmatrix}
    &(-1)^{n+1}\\I_{n-1}
\end{pmatrix}\), which gives rise to a hyper Kloosterman sum, where square root cancellation has been achieved by Pierre Deligne as a consequence of his proof of Riemann Hypothesis for varieties over finite fields.

The other extreme case for Weyl element is \(w=w_0=\begin{pmatrix}
    &&&&(-1)^{n+1}\\&&&\iddots\\&&1\\&-1\\1
\end{pmatrix}\), the long Weyl element. We know specialize to this case and assume that \(\Gamma=\text{SL}(n,\mathbb{Z})\). Henceforth, whenever \(w\) is the long Weyl element, we will write \(S_n(M,N;D,w)\) as \(S(M,N;D)\). In this case, \(U_w=U\), i.e., the entire unipotent radical survives the intersection. Before we give our expression for this Kloosterman sum, let us recall the Cauchy-Binet Formula.

\begin{lemma}[See \cite{HJ12}]
    \label{CauchyBinet}
    \textbf{(Cauchy-Binet Formula)}
    Let \(A,\ B\in M(n,\mathbb{R})\) and \(C=AB\). \(I,\ J \) be ordered collections of \(k\) rows and columns, respectively. Then
    \[
        \Lambda_I^J(C)=\sum_{K}\Lambda_I^K(A)\Lambda_K^J(B)
    \]
    where \(K\) runs over all ordered collections of \(k\) many strictly increasing column indices of A (and correspondingly rows of \(B\)).
\end{lemma}

Using this as the key ingredient, we can give an alternate expression for \(S_n(M,N;D)\).

\begin{proposition}\label{mainformula}
For \(D\in \mathbb{\mathbb{N}}^{n-1}\), let \(\mathbb{Z}_D=\mathbb{Z}/D_1\mathbb{Z} \times \dots \times \mathbb{Z}/D_{n-1}\mathbb{Z}\) and
\[
\tilde{X}_n(D) = \left\{ (x, y)\in \mathbb{Z}_D^2\middle|\ \exists\ \gamma \in SL(n, \mathbb{Z}) \text{ with } \begin{array}{r} \Lambda_{[n-i+1]_n}^{[i]}(\gamma) = D_{i}, \\ \Lambda^{[i]}_{n-i,[n-i+2]_n}(\gamma) \equiv x_i \pmod{D_{i}}, \\ \Lambda_{[n-i+1]_n}^{[i-1],i+1}(\gamma) \equiv y_i \pmod{D_{i}} \end{array}\ i=1,\cdots,n-1 \right\}
\]
Then,\begin{equation}
S_n(M,N;D)=\sum_{(x,y)\in \tilde{X}_n(D)}e\left(\sum_{i=1}^{n-1}\frac{M_ix_{i}+N_iy_{i}}{D_i}\right)
\label{formula}    
\end{equation}

\end{proposition}
\begin{proof}
Let \(c=\begin{pmatrix}
    1/D_{n-1}&&&\\
    &D_{n-1}/D_{n-2}&&\\
    &&\ddots&&\\
    &&&D_2/D_1&\\
    &&&&D_1
\end{pmatrix}\),\ \(w_0=\begin{pmatrix}
    &&&&(-1)^{n+1} \\&&&\iddots\\&&1\\&-1\\1
\end{pmatrix}\)
Consider \(b_1=\begin{pmatrix}
    1&\alpha_{1,2}&\cdots&\alpha_{1,n-1}&\alpha_{1,n}\\&1&\cdots&\alpha_{2,n-1}&\alpha_{2,n}\\&&\ddots&\vdots&\vdots\\&&&1&\alpha_{n-1,n}\\&&&&1
\end{pmatrix},b_2=\begin{pmatrix}
    1&\beta_{1,2}&\cdots&\beta_{1,n-1}&\beta_{1,n}\\&1&\cdots&\beta_{2,n-1}&\beta_{2,n}\\&&\ddots&\vdots&\vdots\\&&&1&\beta_{n-1,n}\\&&&&1
\end{pmatrix}\in U(\mathbb{R})\)
and \(\gamma\in SL(n,\mathbb{Z)}\) such that \(\gamma=b_1cw_0b_2\).

Post-multiplying by \(c\) scales the columns of \(b_1\), while pre-multiplying by \(w_0\) permutes (and changes the signs) of the rows of \(b_2\), hence
\[
\gamma=\begin{pmatrix}
    \frac{1}{D_{n-1}}&\frac{D_{n-2}\alpha_{1,2}}{D_{n-1}}&\cdots&\frac{D_{2}\alpha_{1,3}}{D_{1}}&D_1\alpha_{1,n}\\&\frac{D_{n-2}}{D_{n-1}}&\cdots&\frac{D_{2}\alpha_{2,3}}{D_{1}}&D_1\alpha_{2,n}\\&&\ddots&\vdots&\vdots\\&&&\frac{D_2}{D_1} &D_1\alpha_{n-1,n}\\&&&&D_1
\end{pmatrix}\begin{pmatrix}
    &&&&(-1)^{n+1}\\&&&(-1)^n&(-1)^n\beta_1\\&&\iddots&\vdots&\vdots\\&-1&-\beta_{2,3}&\cdots&-\beta_{2,n}\\1&\beta_{1,2}&\beta_{1,3}&\cdots&\beta_{1,n}
\end{pmatrix}=b_1'b_2'
\]

Now, since \(\gamma\) is expressed as the product of two matrices, we can use Cauchy-Binet formula to calculate the minors of \(\gamma\).
\[
\Lambda_{[n-i+1]_n}^{[i]}(\gamma)=\sum_{\substack{|J|=i}}\Lambda_{[n-i+1]_n}^{J}(b_1')\Lambda_J^{[i]}(b_2')
\]
Observe that for the bottom-most \(i\) rows of \(b_1'\), only the right-most \(i\) columns are non-zero. Thus for any ordered collection \(J\) of \(i\) columns, the determinant \(\Lambda_{[n-i+1]_n}^{J}(b_1')\) is zero unless \(J=[n-i+1]_n\). In that case, \(\Lambda_{[n-i+1]_n}^{[n-i+1]_n}(b_1')\) and \(\Lambda_{[n-i+1]_n}^{[i]}\) can be evaluated as a determinant of an upper-triangular and a lower-right-triangular matrix respectively,
\[
\Lambda_{[n-i+1]_n}^{[n-i+1]_n}(b_1')=\frac{D_{i}}{D_{i-1}}\Lambda_{[n-i]_n}^{[n-i]_n}=\frac{D_i}{D_{i-1}}\frac{D_{i-1}}{D_{i-2}}\Lambda_{[n-i-1]_n}^{[n-i-1]_n}=\cdots=\frac{D_i}{D_{i-1}}\times\frac{D_i}{D_{i-1}}\times\cdots\times\frac{D_2}{D_{1}}\times D_1=D_i 
\]
\[
\Lambda_{[n-i+1]_n}^{[i]}(b_2')=(-1)^{i}(-1)^i\Lambda_{[n-i+2]_n}^{[i-1]}=\Lambda_{[n-i+2]_n}^{[i-1]}\cdots=1
\]
Hence, \(\Lambda_{[n-i+1]_n}^{[i]}(\gamma)=D_{n-i+1}\).
Similarly,
\[
\Lambda_{[n-i+1]_n}^{[i-1],i+1}(\gamma)=\sum_{\substack{|J|=i}}\Lambda_{[n-i+1]_n}^{J}(b_1')\Lambda_J^{[i-1],i+1}(b_2')
\]
Again, the only non-zero term in the sum corresponds to \(J=[n-i+1]_n\), in that case we can similarly calculate,
\[
\Lambda_{[n-i+1]_n}^{[i-1][i+1]}(b'_2)=\beta_{i,i+1}
\]
And therefore,
\[
\beta_{i,i+1}=\frac{\Lambda_{[n-i+1]_n}^{[i-1],i+1}(\gamma)}{D_i}
\]
Similarly,
\[
\alpha_{i,i+1}=\frac{\Lambda_{n-i,[n-i+2]_n}^{[i]}(\gamma)}{D_i}
\]
Hence,
\[
S(M,N;D)=\sum_{\gamma\in U(\mathbb{Z})\backslash C(c)/U(\mathbb{Z})}e\left(\sum_{i=1}^{n-1}\frac{M_i\Lambda_{n-i,[n-i+2]_n}^{[i]}(\gamma)+N_i\Lambda_{[n-i+1]_{n}}^{[i-1],i+1}(\gamma)}{D_i}\right)
\]
Now since, \(\gamma\) varies over a double coset of \(U(\mathbb{Z})\), the \(\Lambda_{[n-i+1]_n}^{[i-1],i+1}(\gamma),\ \Lambda_{n-i,[n-i+2]_n}^{[i]}(\gamma)\) vary modulo \(D_i\). If we collect all the possible values these \(2i\) determinants can take in \(\mathbb{Z}_D^2\), the resulting set is precisely \(\tilde{X}_n(D)\). Hence, the Kloosterman sum takes the form,
\[
S_n(M,N;D)=\sum_{(x,y)\in \tilde{X}_n(D)}e\left(\sum_{i=1}^{n-1}\frac{M_ix_{i}+N_iy_{i}}{D_i}\right)
\]

\end{proof}

\section{Properties}
We proceed to use \cref{mainformula} to prove the generalizations of the properties that were proved in \cite{BFG88}. Some of these properties were also pointed out in \cite{Fr87}. The first observation is immediate.
\begin{proposition}
    The value of \(S_n(M,N;D)\) depends only on the congruence class of each \(M_i\) and \(N_i\) in \(\mathbb{Z}/D_i\mathbb{Z}\).
\end{proposition}
\begin{proof}
    This is immediate from \cref{formula} and because \(e(z)\) is periodic modulo \(1\).
\end{proof}

We would use the following lemma about the anti-transpose of a matrix to prove that the role of \(M\) and \(N\) can be interchanged. 

\begin{lemma}
    Given any matrix \(A\in M_n(\mathbb{R)}\), \(det(A)=det(A^{at})\)\label{det_at}.
\end{lemma}
\begin{proof}
    Let \(v_o=\begin{pmatrix}
        &&&1\\&&\iddots\\&1\\1
    \end{pmatrix}\) be the ``unsigned'' long Weyl element in \(GL(n,\mathbb{R})\). It is easy to see that
    \[
    A^{at}=v_o A^t v_o
    \]
    Then, it follows that \(det(A^{at})=det(v_o)^2 det(A^t)=det(A)\).
\end{proof}

\begin{proposition}
    \(S_n(M,N;D)=S_n(N,M;D)\)
\end{proposition}
\begin{proof}
    From \eqref{formula}, it is enough to show that if \((x,y)\in \tilde{X}_{n}(D)\), then \((y,x)\in \tilde{X}_{n}(D)\). Let \(\gamma\in SL(n,\mathbb{Z})\) be the matrix corresponding to \((x,y)\in \mathbb{Z}_D^2\). We claim that \(\gamma^{at}\in SL(n,\mathbb{Z})\) which corresponds to \((y,x)\in \mathbb{Z}_D^2\).

    Clearly, by \cref{det_at}, \(det(\gamma^{at})=det(\gamma)=1\), and the entries of \(\gamma^{at}\) are integers, so \(\gamma^{at}\) indeed belongs to \(SL(n,\mathbb{Z})\). Note that 
    \[\Lambda_{[n-i+1]_n}^{[i]}(\gamma^{at})=det({\mathcal{M}_{[n-i+1]_n}^{[i]}(\gamma^{at}}))=det((\mathcal{M}_{[n-i+1]_n}^{[i]}(\gamma))^{at})=det(\mathcal{M}_{[n-i+1]_n}^{[i]}(\gamma))=D_{i}\]
    \begin{align*}
       \Lambda^{[i]}_{n-i,[n-i+2]_n}(\gamma^{at})&=det((\mathcal{M}_{n-i,[n-i+2]_n}^{[i]}(\gamma))^{at})=det((\mathcal{M}_{[n-i+1]_n}^{[i-1],i+1}(\gamma))^{at})\\&=\Lambda^{[i-1],i+1}_{[n-i+1]_n}(\gamma)\equiv y_i\pmod{D_{i}}
    \end{align*}
    Similarly, 
    \[ 
    \Lambda^{[i-1],i+1}_{[n-i+1]_n}(\gamma^{at})=\Lambda^{[i]}_{n-i,[n-i+2]_n}(\gamma^{at})\equiv x_i\pmod{D_{i}}
    \]
    Thus, \(\gamma^{at}\) satisfies the criteria for \((y,x)\) to be in \(\tilde{X}_n(D)\). 
\end{proof}

Our next target is to prove that we can reverse the co-ordinates of each \(M,N\) and \(D\) without changing the value of \(S(M,N;D)\). For this we would need to use another well-known result from linear algebra relating the minors of a matrix to the minors of it's inverse.  

\begin{lemma}[See \cite{HJ12}]
    \label{Jacobi}
    \textbf{(Jacobi's Complementary Minor Formula)}
    
    Let \(A\in GL(n,\mathbb{R})\), \(I,\ J \) be ordered collections of \(k\) rows and columns, respectively, let \(I',\ J'\) be the complements of \(I,\ J\), respectively, and let \(\sum I=\sum_{i\in I}i\), \(\sum J=\sum_{j\in J}j\). Then 
    \[
    \Lambda_{I}^J(A)=(-1)^{\sum I + \sum J}det(A)\Lambda_{I'}^{J'}(A^{-1})
    \]
\end{lemma}
If \(A\in SL(n,\mathbb{R})\), the \(det(A)\) in the above expression will not be present, but we will see in the next proposition that we need a matrix, where the statement would be true without the \(-1^{\sum I+\sum J}\) sign factor. For this, we would use the transpose of the matrix of minors of \(A\), which we denote as \(A^{-}\). The \((i,j)^\text{th}\) element of \(A^{-}\) is the \((j,i)^\text{th}\) minor of A given by \(\Lambda_{[n]^{(j)}}^{[n]^{(i)}}(A)\). This is almost equal to \(A^{-1}\), in fact,
\begin{equation}
    A^{-1}_{i,j}=(-1)^{\sum I+\sum J} A^{-}_{i,j} \label{comple}
\end{equation}
\begin{corollary}
    Let \(A\in SL(n,\mathbb{R})\) and \(A^-\) be the transposed matrix of minors of \(A\). Let \(I,\ J \) be ordered collections of \(k\) rows and columns, respectively, let \(I',\ J'\) be the complements of \(I,\ J\), respectively, and let \(\sum I=\sum_{i\in I}i\), \(\sum J=\sum_{j\in J}j\). Then 
    \[
    \Lambda_{I}^J(A)=\Lambda_{J'}^{I'}(A^{-})
    \]
\end{corollary}
\begin{proof}
    This follows directly from \cref{Jacobi} and \cref{comple}.
\end{proof}
\begin{proposition}
    \label{tilde}
    \(S_n(M,N;D)=S_n(\tilde{M},\tilde{N};\tilde{D})\)
\end{proposition}
\begin{proof}
It is enough to prove that if \((x,y)\in \tilde{X}_n(D)\) then \((\tilde{x},\tilde{y})\in \tilde{X}_n(\tilde{D})\). Let \(\gamma\in SL(n,\mathbb{Z})\) be the matrix corresponding to \((x,y)\). We claim \(\gamma^-\in SL(n,\mathbb{Z})\) corresponds to \(\tilde{x},\tilde{y}\in\tilde{X}_n(\tilde{D})\).

The entries of \(A^-\) are minors of \(A\), and hence are integers. It is easy to see that \(det(A^-)=det(Ad(A))=1\), so \(A^-\in SL(n,\mathbb{Z})\). Applying \cref{Jacobi} with \(I=[n-i+1]_n, J=[i]\) gives
\[
    \Lambda_{[n-i+1]_n}^{[i]}(\gamma^{-})=\Lambda_{[i+1]_n}^{[n-i]}(\gamma)=D_{n-i}=\tilde{D}_i
\]
\[
    \Lambda_{n-i,[n-i+2]_n}^{[i]}(\gamma^{-})=\Lambda_{[i+1]_n}^{[n-i-1],n-i+1}(\gamma)\equiv y_{n-i}\equiv \tilde{y}_i \pmod{\tilde{D}_i}
\]
Similarly,
\[
    \Lambda_{[n-i+1]_n}^{[i-1],i+1}(\gamma^{-})=\Lambda_{i,[i+2]_n}^{[n-i]}(\gamma)\equiv x_{n-i}\equiv \tilde{x}_i \pmod{\tilde{D}_i}
\]
Thus, \(\gamma^{-}\) satisfies the criteria for \((\tilde{x},\tilde{y})\) to be in \(\tilde{X}_n(\tilde{D})\). 
\end{proof}
We next prove the degeneracy result that if certain moduli \(D_i\)'s are zero then the Kloosterman sum reduces to a lower rank Kloosterman sum. First we assume only the last \(D_i\) is zero, in that case the last component of \(M\) and \(N\) gets removed in the resulting Kloosterman sum.
\begin{proposition}
    \label{deg1}
    If \(D_{n-1}=1\) then \(S_n(M,N;D)=S_{n-1}(M^{(n-1)},N^{(n-1)};D^{(n-1)})\)
\end{proposition}
\begin{proof}
    It is enough to show that \((x,y)\in \tilde{X}_n(D)\) if and only if \((x^{(n-1)},y^{(n-1)})\in \tilde{X}_{n-1}(D^{(n-1)})\).
    If \((x,y)\in \tilde{X}_n(D)\), let \(\gamma\in SL(n,\mathbb{Z)}\) correspond to \((x,y)\) for \(\tilde{X}_n(D)\). Consider the lower bottom sub-matrix \(\rho=\mathcal{M}_{[2]_n}^{[n-1]}(\gamma)\). \(\rho\) has integer entries because \(\gamma\) has integer entries, and
    \[
    det(\rho)=\Lambda_{[2]_n}^{[n-1]}(\gamma)=D_{n-1}=1
    \]
    Hence \(\gamma\in SL(n,\mathbb{Z})\). Also, we have
    \[
    \Lambda_{[n-i]_{n-1}}^{[i]}(\rho)=\Lambda_{[n-i+1]_{n}}^{[i]}(\gamma)=D_i\qquad 1\le i \le n-1
    \]
    \[
    \Lambda_{n-i-1,[n-i+1]_{n-1}}^{[i]}(\rho)=\Lambda_{n-i,[n-i+2]_{n}}^{[i]}(\gamma)\equiv x_i\pmod{D_i}\qquad 1\le i \le n-1
    \]
    \[
    \Lambda_{[n-i]_{n-1}}^{[i-1],i+1}(\rho)=\Lambda_{[n-i+1]_{n}}^{[i-1],i+1}(\gamma)\equiv y_i\pmod{D_i}\qquad 1\le i \le n-1
    \]
    Therefore, \(\rho\) corresponds to \((x^{(n-1)},y^{(n-1)})\in \tilde{X}_{n-1}(D^{(n-1)})\).

    Conversely, let \((x^{(n-1)},y^{(n-1)})\in \tilde{X}_{(n-1)}(D^{(n-1)})\). Consider the block matrix \(\rho=\begin{pmatrix}
        &(-1)^{n+1}\\\gamma
    \end{pmatrix}\)
    Clearly, \(\rho\) has integer entries and \(det(\rho)=(-1)^{2(n+1)}det(\gamma)=1\), hence \(\rho\in SL(n,\mathbb{Z})\). Also,
    \[
    \Lambda_{[n-i+1]_n}^{[i]}(\rho)=\Lambda_{[n-i]_{n-1}}^{[i]}(\gamma)=D_i\qquad 1\le i \le n-1
    \]
    \[
    \Lambda_{n-i,[n-i+2]_n}^{[i]}(\rho)=\Lambda_{n-i-1,[n-i+1]_{n-1}}^{[i]}(\gamma)\equiv x_i \pmod{D_i}\qquad 1\le i \le n-1
    \]
    \[
    \Lambda_{[n-i+1]_n}^{[i-1],i+1}(\rho)=\Lambda_{[n-i]_{n-1}}^{[i-1],i+1}(\gamma)\equiv y_i \pmod{D_i}\qquad 1\le i \le n-1
    \]
    It remains to check that the conditions hold for \(i=n-1\). Indeed,
    \(\Lambda_{[2]_{n}}^{[n-1]}(\rho)=det(\gamma)=1=D_{n-1}\)
    The conditions for \(x_{n-1}\) and \(y_{n-1}\) are congruences modulo \(D_{n-1}\), but \(D_{n-1}\) is 1, and hence the congruences are trivially satisfied. Thus \(\rho\) corresponds to \((x,y)\in \tilde{X}_n(D)\).
\end{proof}
We now state other cases of degeneracy in the following corollary to the above proposition.
\begin{corollary} For \(M,N,D \in \mathbb{Z}^{n-1}\),
\begin{enumerate}[(\roman*)]
    \item If \(D_1=1\), then \(S_n(M,N;D)=S_{n-1}(M^{(1)},N^{(1)};D^{(1)})\)
    \item If \(D_{i},\cdots, D_{n-1}\) are all equal to \(1\), then 
    \[S_n(M,N;D)=S_{i}((M_1,\cdots,M_{i-1}),(N_1,\cdots,N_{i-1});(D_1,\cdots,D_{i-1}))\]
    \item If \(D_{1},\cdots, D_{i}\) are all equal to \(1\), then 
    \[S_n(M,N;D)=S_{n-i-1}((M_{i+1},\cdots,M_{n-1}),(N_{i+1},\cdots,N_{n-1});(D_{i+1},\cdots,D_{n-1}))\]
    \item If \(D_{1},\cdots,D_{i}\) and \(D_{j},\cdots,D_{n-1}\) are all equal to \(1\), then
    \[
    S_{n}(M,N;D)=S_{j-i-1}((M_{i+1},\cdots,M_{j-1}),(N_{i+1},\cdots,N_{j-1}),(D_{i+1},\cdots,D_{j-1}))
    \]
    \item If \(D_j=1\ \forall j\ne i\), then \(S_n(M,N;D)=S_2(M_i,N_i;D_i)\)
\end{enumerate}
\end{corollary}
\begin{proof}
Statement \((i)\) is a direct consequence of \cref{deg1} and \cref{tilde}. Statements \((ii)\) and \((iii)\) are obtained inductively using \cref{deg1} and statement \((i)\), respectively. Statement \((iv)\) is obtained by using \((ii)\) and \((iii)\) together. Lastly, statement \((v)\) follows from \((iv)\) by taking j=i+2.
\end{proof}

We now prove one of the most important properties of Kloosterman sums, which is the factorization property for co-prime moduli.

\begin{proposition}\label{factor}
    If \(gcd(D_1\cdots D_{n-2},D_{n-1})=1\) then,  
    \[
    S_n(M,N;D)=S_{n-1}(M',N^{(n-1)};D^{(n-1)})\times S_{2}(D_{n-2}M_{n-1},N_{n-1};D_{n-1})
    \]
    where \(M'=(M_1,\cdots,M_{n-3}, D_{n-1}M_{n-2})\).
\end{proposition}
\begin{proof}
Let \(pD_1D_2\cdots D_{n-2}+p'D_{n-1}=1\). 

We proceed by defining a function \(\phi:\tilde{X}_{n-1}(D^{(n-1)})\times\tilde{X}_2(D_{n-1})\to\tilde{X}_n(D)\) by
\[
((x,y),(u,v))\mapsto((x_1,\cdots,x_{n-3},D_{n-1}x_{n-2},D_{n-2}u),(y_1,\cdots,y_n,v))
\]
First, we show that this map is well-defined and bijective. Let \(\gamma\in SL_{n-1}(\mathbb{Z)}\) correspond to \((x,y)\) in \(\tilde{X}_{(n-1)}(D^{(n-1)})\) and \(\gamma'\in SL(2,\mathbb{Z})\) correspond to \((u,v)\) in \(\tilde{X}_{2}(D_{n-1})\). Write \(\gamma\) as the block matrix \(\begin{pmatrix}
    a&b\\c&d
\end{pmatrix}\), where \(c\in M(n-2,\mathbb{Z})\). Consider the block matrix
\[
\rho=\begin{pmatrix}
     &(-1)^{n}\gamma'_{1,1}&(-1)^n(pD_1D_2\cdots D_{n-2}\gamma'_{1,2}-p') \\
    D_{n-1}a&D_{n-1}b&(-1)^{n}pD_1D_2\cdots D_{n-3}\gamma'_{2,2} \\
    c&d
\end{pmatrix}
\]
\[
    \Lambda_{[n-i+1]_n}^{[i]}(\rho)=\Lambda_{[n-i]_{n-1}}^{[i]}(c)=\Lambda_{[n-i]_{n-1}}^{[i]}(\gamma)=D_{i}\quad1\le i\le n-2
\]
Similarly,
\[
    \Lambda_{[n-i+1]_n}^{[i-1],i+1}(\rho)=\Lambda_{[n-i]_{n-1}}^{[i-1],i+1}(\gamma)\equiv y_i \pmod{D_{i}}\quad1\le i\le n-2
\]
\[
    \Lambda_{n-i,[n-i+2]_n}^{[i]}(\rho)=\Lambda_{n-i-1,[n-i+1]_{n-1}}^{[i]}(\gamma)\equiv x_i \pmod{D_{i}}\quad1\le i\le n-3
\]
For \(i=n-2\),
\[
\Lambda_{2,[4]_n}^{[n-2]}(\rho)=\Lambda_{1,[3]_{n-1}}^{[n-2]}\begin{pmatrix}
    D_{n-1}a\\c
\end{pmatrix}=D_{n-1}\Lambda_{1,[3]_{n-1}}^{[n-2]}(\gamma)\equiv D_{n-1}x_{n-2} \pmod{D_{i}}
\]
For \(i=n-1\),
\[
\Lambda_{[2]_n}^{[n-1]}(\rho)=det\begin{pmatrix}
    D_{n-1}a& D_{n-1}b\\c &d
\end{pmatrix}=D_{n-1}det\begin{pmatrix}
    a &b\\c&d
\end{pmatrix}= D_{n-1}det(\gamma)=D_{n-1}
\]
\[
\Lambda_{1,[3]_n}^{[n-1]}(\rho)=det\begin{pmatrix}
    & (-1)^{n}\gamma'_{1,1}\\c &d
\end{pmatrix}=\gamma'_{1,1}det(c)=D_{n-2}\gamma'_{1,1}\equiv D_{n-2}u \pmod{D_{n-1}}
\]
\begin{align*}
\Lambda_{[2]_n}^{[n-2],n}(\rho)&=det\begin{pmatrix}
    D_{n-1}a& (-1)^{n}pD_1D_2\cdots D_{n-3}\gamma'_{2,2}\\c &
\end{pmatrix}\\&=pD_1D_2\cdots D_{n-3}\gamma'_{2,2}\times det(c)= pD_1D_2\cdots D_{n-2}\gamma'_{2,2}\equiv y_{n-1}\pmod{D_{n-1}}
\end{align*}

Lastly, \(\rho\) has integer entries and
\begin{align*}
det(\rho)&= (\gamma'_{1,1})(pD_1D_2\cdots D_{n-3}\gamma'_{2,2})D_2-(pD_1D_2\cdots D_{n-2}\gamma'_{1,2}-p')D_{n-1}\\
&=pD_1D_2\cdots D_{n-2}\gamma'_{1,1}\gamma'_{2,2}-pD_1D_2\cdots D_{n-1}\gamma'_{1,2}+p'D_{n-1}\\
&=pD_1D_2\cdots D_{n-2}(\gamma'_{1,1}\gamma'_{2,2}-D_{n-1}\gamma'_{1,2})+p'D_{n-1}\\
&=pD_1D_2\cdots D_{n-2}+p'D_{n-1}=1
\end{align*}
Hence \(\rho\in SL(n,\mathbb{Z})\), and corresponds to \(\phi((x,y))\) in \(\tilde{X}_n(D)\), and \(\phi\) is well-defined.

It is easy to see that \(\phi\) is indeed a bijection. Therefore,
\begin{align*}
    &S_{n}(M,N;D)=\sum_{(a,b)\in \tilde{X}_{n}(D)}e\left(\sum_{i=1}^{n-1}\frac{ M_ia_i+ N_ib_i}{D_i}\right)\\
    &=\sum_{\substack{(x,y)\in\tilde{X}_{n-1}(D^{(n-1)})\\(u,v)\in \tilde{X}_{2}(D_{n-1})}}e\left(\sum_{i=1}^{n-3}\frac{M_ix_i+N_iy_i}{D_i}+\frac{M_{n-2}D_{n-1}x_{n-2}+N_{n-2}y_{n-2}}{D_{n-2}}+\frac{M_{n-1}D_{n-2}x_{n-1}+N_{n-1}y_{n-1}}{D_{n-1}}\right)\\
    &=\left(\sum_{(x,y)\in\tilde{X}_{n-1}(D^{(n-1)})}e\left(\sum_{i=1}^{n-3}\frac{M_ix_i+N_iy_i}{D_i}+\frac{M_{n-2}D_{n-1}x_{n-2}+N_{n-2}y_{n-2}}{D_{n-2}}\right)\right)\times\\&\qquad\qquad\qquad\left(\sum_{(u,v)\in \tilde{X}_{2}(D_{n-1})}e\left(\frac{M_{n-1}D_{n-2}x_{n-1}+N_{n-1}y_{n-1}}{D_{n-1}}\right)\right)\\
    &=S_{n-1}(M',N^{(n-1)};D^{(n-1)})\times S_{2}(D_{n-2}M_{n-1},N_{n-1};D_{n-1})
\end{align*}
\end{proof}
\begin{corollary}
    If \(gcd(D_1,D_2\cdots D_{n-1})=1\) then,  
    \[
    S_n(M,N;D)=S_{2}(D_2M_{1},N_{1};D_{1})\times S_{n-1}(M'',N^{(1)};D^{(1)})
    \]
    where \(M'=(D_1M_2,M_3,\cdots,M_{n-1})\)
\end{corollary}
\begin{proof}
    This follows directly from \cref{factor} and \cref{tilde}.
\end{proof}
As a corollary, if no two \(D_i\)'s have the same prime factor, then the higher rank Kloosterman sum decomposes into \(n-1\) simpler classical Kloosterman sums.
\begin{corollary}
If \(D_1,\cdots,D_{n-1}\) are pairwise co-prime, then 
\[
S_n(M,N,D)=\prod_{i=1}^{n-1}S_2(D_{i-1}M_iD_{i+1},N_i,D_i)
\]
with the convention that \(D_0=1,D_n=1\)
\end{corollary}
\begin{proof}
    This follows inductively by \cref{factor}.
\end{proof}
\section{Acknowledgement}
The author thanks Professor Ritabrata Munshi for suggesting the problem and for his helpful discussions. In addition, the author acknowledges the Indian Statistical Institute, Kolkata for providing a supportive research environment.

\vspace{1cm}

\begin{thebibliography}{99}

\bibitem[BFG88]{BFG88}Daniel Bump, Solomon Friedberg, and Dorian Goldfeld, Poincar\'e series and Kloosterman sums for \(SL(3,\mathbb{Z})\), Acta Arith. 50 (1988), no. 1, 31–89

\bibitem[G06]{G06}Goldfeld D. Automorphic Forms and L-Functions for the Group GL(n,R). Cambridge University Press; 2006.

\bibitem[BM24]{BM24}Blomer, V., Man, S.H. Bounds for Kloosterman sums on GL(n)
. Math. Ann. 390, 1171–1200 (2024)

\bibitem[HJ12]{HJ12}Horn, Roger A. and Johnson, Charles R.,Matrix Analysis 2nd Edition, Cambridge University Press 2012

\bibitem[Fr87]{Fr87} Friedberg, S. Poincaré series for GL(n): Fourier expansion, kloosterman sums, and algebreo-geometric estimates. Math Z 196, 165–188 (1987)
\end{thebibliography}
\end{document}